\documentclass[11pt]{amsart}
\usepackage{geometry}                % See geometry.pdf to learn the layout options. There are lots.
\usepackage{graphicx}
\usepackage{amssymb}
\usepackage{epstopdf}
\newtheorem{theorem}{Theorem}[section]

\def \r{{\bf H}_{\R}}
\def\z{{\bf z}}
\def\w{{\bf w}}
\def\P{\mathbb P}
\def\R{\mathbb R}

\def\C{\mathbb C}
\def \c{{\bf H}_{\C}}

\def\PSL{\mathrm{PSL}}
\def\PU{\mathrm{PU}}

\newtheorem{corollary}[theorem]{Corollary}
\newtheorem{proposition}[theorem]{Proposition}
\newtheorem{lemma}[theorem]{Lemma}

\title{Free groups generated by two parabolic maps}
\author{Sagar B.~Kalane and John R.~Parker}

\address{Department of Mathematical Sciences, Indian Institute of Science Education and Research (IISER) Pune,
	Dr Homi Bhabha Rd, Ward No. 8, NCL Colony, Pashan, Pune, Maharashtra 411008}
\email{sagark327@gmail.com, sagarkalane@iiserpune.ac.in}

\address{Department of Mathematical Sciences, Durham University, Upper Mountjoy, South Road, Durham DH1 3LE, UK}
\email{j.r.parker@durham.ac.uk}
\subjclass[2020]{Primary 22E40; Secondary  51M10, 32M15, 20H10}

\keywords{Complex hyperbolic geometry, Heisenberg translations, Discreteness.}

\date{\today}
\thanks{First author is supported by NBHM post doctoral fellowship}

\begin{document}
	\maketitle
	
\begin{abstract}
In this paper we consider a group generated by two unipotent parabolic elements of ${\rm SU}(2,1)$ with distinct fixed points.
We give several conditions that guarantee the group is discrete and free. We also give a result on the diameter of a
finite ${\mathbb R}$-circle in the Heisenberg group.
\end{abstract}

\section{Introduction}
The study of free, discrete groups has a long history dating back to Schottky and Klein in the nineteenth century. We will be
particularly interested in groups generated by two unipotent parabolic maps in ${\rm SU}(2,1)$ and their action on complex 
hyperbolic space and its boundary. The conditions we give could be thought of as complex hyperbolic analogues of the
results proved by Lyndon and Ullman \cite{lyu} and by Ignatov \cite{ig78} giving conditions under which two parabolic elements of
${\rm PSL}(2,{\mathbb C})$ generate a free Kleinian group. Our work is very closely related to the well-known Riley slice of 
Schottky space, where Riley considered the space of conjugacy classes of subgroups of $\PSL(2,\C)$ generated by two 
non-commuting parabolic maps, see in \cite{ks}. In \cite{PW} Parker and Will considered a related problem, namely they also 
studied groups with two unipotent generators, but they made the additional assumption that the product of these maps is 
also unipotent. We will comment on the relationship between our results and those in \cite{PW} below.

The main theme of the paper concerns groups generated by two Heisenberg translations with distinct fixed points. We normalise
so that the fixed points are $\infty$ and $o$, the origin in the Heisenberg group. Specifically, we consider the group generated by 
	\begin{equation}\label{eq-A-B}
		A=\left(\begin{matrix} 1 & -\sqrt{2}s_1e^{-i\theta_1} & -s_1^2+it_1 \\ 0 & 1 & \sqrt{2}s_1e^{i\theta_1} \\ 
		0 & 0 & 1 \end{matrix}\right),
		\quad
		B=\left(\begin{matrix} 1 & 0 & 0 \\ \sqrt{2}s_2e^{i\theta_2} & 1 & 0 \\ -
		s_2^2+it_2 & -\sqrt{2}s_2e^{-i\theta_2} & 1 \end{matrix}\right).
	\end{equation}
From this, it might appear that the space of such pairs of transformations has dimension six, and is parameterised
by $s_j,\,t_j,\,\theta_j$ for $j=1,\,2$. In fact, it has dimension four.
There is a further normalisation we can do using the stabiliser of the pair $\{o,\infty\}$. This depends on 
$(k,\psi)\in{\mathbb R}_+\times [0,2\pi)$ which act as follows:
\begin{equation}\label{eq-action-stabiliser}
(s_1,t_1,\theta_1;s_2,t_2,\theta_2)\longmapsto (s_1k, t_1k^2, \theta_1+\psi; s_2/k, t_2/k^2,\theta_2+\psi).
\end{equation}
In order to show the symmetry in the parameters, we choose not to make this normalisation in the statement of the results. But 
in some of the proofs we use it to simplify calculations, for example by choosing $\theta_2=0$. 
	
We want to find conditions on $s_j$, $t_j$, $\theta_j$ that ensures $\langle A,B\rangle$ is discrete and freely generated by 
$A$ and $B$. To do
so, we will use Klein's combination theorem (sometimes called the ping-pong theorem) or variants of it on the boundary of complex
hyperbolic space. This is given in the following proposition.

\begin{proposition}{\label{KCT}}
	Let $A$ and $B$ be the Heisenberg translations fixing $\infty$ and $o$ respectively, given by \eqref{eq-A-B}. 
	If the fundamental domains
	$D_A\subset \partial{\bf H}^2_{\mathbb C}$ for $\langle A\rangle$ and $D_B\subset \partial{\bf H}^2_{\mathbb C}$  for 
	$\langle B\rangle$ satisfy $D_A^\circ\cap D_B^\circ\neq \emptyset$ and 
	$\overline{D}_A\cup\overline{D}_B=\partial{\bf H}^2_{\mathbb C}$, then $\langle A,B \rangle$ is free and discrete.
\end{proposition}

Our results are also related to proofs of discreteness of complex hyperbolic isometry groups using other variations on Klein's 
combination theorem. For example, see Goldman and Parker \cite{gp1}, Wyss-Gallifent \cite{wg}, Monaghan, Parker and
Pratoussevitch \cite{mpa} or Jiang and Xie \cite{xyj}.

Our main theorem is:

\begin{theorem}\label{thm-main0}
Let $(s_1e^{i\theta_1},t_1)$ and $(s_2e^{i\theta_2},t_2)$ be non-trivial 
elements of the Heisenberg group. Here, $\theta_1$ and $\theta_2$
are only defined when $s_1\neq 0$ and $s_2\neq 0$. Let $A$ and $B$ given by \eqref{eq-A-B} be the 
associated Heisenberg translations fixing $\infty$ and $o$ respectively. Replacing one of these by its 
inverse if necessary, we suppose $-\pi/2\le (\theta_1-\theta_2)\le \pi/2$. 
If one of the following four conditions is satisfied then
$\langle A,\,B\rangle$ is discrete and freely generated by $A$ and $B$. The conditions are
\begin{enumerate}
\item[(1)] 
\begin{eqnarray*} 
|s_1^2+it_1|^{1/2}|s_2^2+it_2|^{1/2} & \ge & 
2^{1/2}\left(\left(1-\frac{t_2}{|s_2^2+it_2|}\right)^{1/3}+\left(1+\frac{t_2}{|s_2^2+it_2|}\right)^{1/3}\right)^{3/4}\\
&& \quad \times \left(\left(1-\frac{t_1}{|s_1^2+it_1|}\right)^{1/3}+\left(1+\frac{t_1}{|s_1^2+it_1|}\right)^{1/3}\right)^{3/4},
\end{eqnarray*}
\item[(2)] if $s_1\neq0$ then $\begin{displaystyle}
s_1\,|s_2^2+it_2|^{1/2} \ge \begin{cases}
\begin{displaystyle}\frac{2~{s_2}^3}{|s_2^2+it_2|^{3/2}}\,\cos(\theta_1-\theta_2)+2  \end{displaystyle} & \hbox{if } s_2\neq0, \\
2 & \hbox{if } s_2=0.
\end{cases}
\end{displaystyle}$
\item[(3)] if $s_2\neq 0$ then $\begin{displaystyle}
|s_1^2+it_1|^{1/2}\,s_2 \ge \begin{cases}
\begin{displaystyle} \frac{2~{s_1}^3}{|s_1^2+it_1|^{3/2}}\,\cos(\theta_1-\theta_2)+2 \end{displaystyle} & \hbox{if } s_1\neq0, \\
2 & \hbox{if } s_1=0.
\end{cases}
\end{displaystyle}$
\item[(4)] if both $s_1,\,s_2\neq 0$ then $s_1s_2 \ge 4\cos^3\bigl((\theta_1-\theta_2)/3\bigr)$.
\end{enumerate}
\end{theorem}

Observe that the expressions above are all invariant under the action of maps that fix both fixed points. Specifically,
using the action of $(k,\psi)\in{\mathbb R}_+\times [0,2\pi)$ from \eqref{eq-action-stabiliser} the left hand side in each case
is a product of two terms, one scaling by $k$ and the other by $1/k$. Similarly, each term on the right hand side involving $(s_j,t_j)$
does not change
when we scale by $k$ and the only place $\theta_1$ and $\theta_2$ arise is via $(\theta_1-\theta_2)$, which does not change
when we add $\psi$ to both angles.

Note that, each item in the above theorem involves different techniques of the proof. The first item of the theorem follows 
by considering fundamental domains bounded by Cygan spheres (special cases of bisectors), whereas last item follows by 
considering fundamental domains bounded by two fans. The middle two items have a mix of Cygan spheres and fans.

Substituting $s_1=s_2=0$ in Theorem~\ref{thm-main0} (1) we 
obtain the following corollary,  which is well known (for example, it is implicit in Section 3 of \cite{par1} and it is written down 
explicitly in Theorem 1.1 of Xie, Wang and Jiang \cite{xwy}). 

\begin{corollary}\label{thm-main2}
Let $(0,t_1)$ and $(0,t_2)$ be elements of the Heisenberg group.  Let $A$ and $B$ given by \eqref{eq-A-B} be the 
associated (vertical) Heisenberg translations fixing $\infty$ and $o$ respectively. If 
$|t_1|^{1/2}\,|t_2|^{1/2} \ge 2$ then $\langle A,\,B\rangle$ is discrete and freely generated by $A$ and $B$. 
\end{corollary}

Note that the right hand sides of parts (1), (2) and (3) of Theorem~\ref{thm-main0}  involve $(s_1e^{i\theta_1},t_1)$ and 
$(s_2e^{i\theta_2},t_2)$. By eliminating $s_j$ and $t_j$, we can get weakening of Theorem~\ref{thm-main0} as follows:

\begin{theorem}\label{thm-main1}
Let $(s_1e^{i\theta_1},t_1)$ and $(s_2e^{i\theta_2},t_2)$ be as in Theorem~\ref{thm-main0} and let 
$A$ and $B$ given by \eqref{eq-A-B}. If one of the following three conditions is satisfied then
$\langle A,\,B\rangle$ is discrete and freely generated by $A$ and $B$. The conditions are
\begin{enumerate}
\item[(1')] $\quad \bigl|s_1^2+it_1\bigr|^{1/2}\,\bigl|s_2^2+it_2\bigr|^{1/2}\ge 4$; 
\item[(2')] if $s_1\neq 0$ then $\quad s_1\,\bigl|s_2^2+it_2\bigr|^{1/2}\ge 4\cos^2\bigl((\theta_1-\theta_2)/2\bigr)$; 
\item[(3')] if $s_2\neq 0$ then $\quad \bigl|s_1^2+it_1\bigr|^{1/2}\, s_2 \ge 4\cos^2\bigl((\theta_1-\theta_2)/2\bigr)$.
\end{enumerate}
\end{theorem}

\begin{proof}
First note that for $-1\le x\le 1$ we have 
$$
\bigl((1-x)^{1/3}+(1+x)^{1/3}\bigr)^{3/4}\le 2^{3/4}
$$
with equality if and only if $x=0$. Therefore, (1') follows from (1). Secondly,
$$
 \frac{2s_2^3}{|s_2^2+it_2|^{3/2}}\,\cos(\theta_1-\theta_2)+2 \le 2\cos(\theta_1-\theta_2)+2 
 = 4\cos^2\bigl((\theta_1-\theta_2)/2\bigr).
$$
Thus (2') follows from (2) and similarly (3') follows from (3).
\end{proof}

\medskip 

The following lemma shows that part (4) of Theorem \ref{thm-main0} follows from the other parts. Nevertheless, we will
still include a direct geometrical proof of this in Section~\ref{sec-fd-fans}.

\begin{lemma} 
If the condition of Theorem \ref{thm-main0}(4) holds then the conditions of Theorem \ref{thm-main0}(2) and (3) hold.
\end{lemma}

\begin{proof}
First we claim that if $-\pi/2\le \theta_1-\theta_2 \le \pi/2$ then
$$
 4\cos^3\bigl((\theta_1-\theta_2)/3\bigr) \ge 4\cos^2\bigl((\theta_1-\theta_2)/2\bigr)
$$ 
with equality if and only if $\theta_1=\theta_2$. 
To see this, we define $\phi=(\theta_1-\theta_2)/3\in[-\pi/6,\pi/6]$, and write the right hand side in terms of $\phi$.
$$
4\cos^2(3\phi/2) = 2\cos(3\phi)+2 = 8\cos^3(\phi)-6\cos(\phi)+2.
$$
Therefore
\begin{eqnarray*}
4\cos^3\bigl(\phi\bigr) - 4\cos^2\bigl(3\phi/2\bigr) 
& = & -4\cos^3(\phi)+6\cos(\phi)-2 \\
& = & 2\bigl(1-\cos(\phi)\bigr)\bigl(2\cos^2(\phi)+2\cos(\phi)-1\bigr).
\end{eqnarray*}
The quadratic term is positive when $\cos(\phi)\ge \sqrt{3}/2$ and so for such values of $\phi$, this expression is non-negative
with equality if and only if $\phi=0$, which proves the claim.

Therefore, if Theorem \ref{thm-main0}(4) holds, we have
\begin{eqnarray*}
s_1\,|s_2^2+it_2|^{1/2} & \ge & s_1s_2 \\
& \ge & 4\cos^3\bigl((\theta_1-\theta_2)/3\bigr) \\
& \ge & 4\cos^2\bigl((\theta_1-\theta_2)/2\bigr) \\
& = & 2\cos(\theta_1-\theta_2)+2 \\
& \ge & \frac{2s_2^3}{|s_2^2+it_2|^{3/2}}\cos(\theta_1+\theta_2)+2.
\end{eqnarray*}
Hence Theorem \ref{thm-main0}(2) holds. A similar argument shows that if Theorem \ref{thm-main0}(4) then
Theorem \ref{thm-main0}(3) holds too. 
\end{proof}

\medskip

In order to prove Theorem~\ref{thm-main0} (1) we will use Lemma~\ref{lem-diam-Rcircle} below, which gives the maximum 
Cygan distance between a point on a finite ${\mathbb R}$-circle and any other point on the same ${\mathbb R}$-circle. 
We believe this will be of independent interest.

An ${\mathbb R}$-circle $R$ is the boundary of a totally geodesic Lagrangian subspace of ${\bf H}^2_{\mathbb C}$
and is the fixed point set of an anti-holomorphic involution $\iota_R$ in the isometry group of ${\bf H}^2_{\mathbb C}$. 
A ${\mathbb C}$-circle $C$ is the boundary of a totally geodesic complex line of ${\bf H}^2_{\mathbb C}$.
Thinking of $\partial{\bf H}^2_{\mathbb C}$ as the one point compactification of the Heisenberg group, an ${\mathbb R}$-circle $R$
is called finite if it does not contain the point $\infty$. Finite ${\mathbb R}$-circles are non-planar space curves with interesting
geometric properties; see Goldman \cite{gol}. In particular, each finite ${\mathbb R}$-circle $R$ is a meridian of a Cygan sphere. 
This Cygan sphere is preserved as a set by $\iota_R$ and its centre is $\iota_R(\infty)$. In particular, every point on $R$ is
the same Cygan  distance from $\iota_R(\infty)$ and we call this distance $r$ its radius. Given a point $p$ of $R$ we want to 
find a point $q$ in $R$ that maximises the Cygan distance from $p$ among all points of $R$. We call this distance the diameter 
$d$ of $R$ at the point $p$. It is clear from the triangle inequality $d\le 2r$.
However, the Cygan metric is not a geodesic metric, and so for most points $p$ the diameter $d$ is strictly less than $2r$. 
In order to write points on $R$ in an invariant way, we use the Cartan angular invariant ${\mathbb A}$.
The lemma below gives a precise formula for the diameter of $R$ at $p$. 

\begin{lemma}\label{lem-diam-Rcircle}
Let $R$ be a finite ${\mathbb R}$-circle fixed by the anti-holomorphic involution $\iota_R$. Let $r$ be the Cygan distance
from $\iota_R(\infty)$ to any point of $R$. For $\alpha\in[0,\pi/2]$, let $p_\alpha$ be a point on $R$
with ${\mathbb A}(p_\alpha,\iota_R(\infty),\infty)=2\alpha-\pi/2$. 
Then the maximum Cygan distance from $p_\alpha$ to any other
point of $R$ is given by
$$
d_\alpha(R)=2^{1/2}r\bigl(\cos^{2/3}(\alpha)+\sin^{2/3}(\alpha)\bigr)^{3/4}.
$$
\end{lemma}

Observe that $d_\alpha(R)\le 2r$, which is attained if and only if $\alpha=\pi/4$, that is whenever 
$(p_\alpha, \iota_R(\infty), \infty \big)$ lie on a common $\R$-circle. Also, $d_\alpha(R)\ge \sqrt{2}r$, which is
attained if and only if $\alpha=0$ or $\alpha=\pi/2$, that is whenever $(p_\alpha, \iota_R(\infty), \infty \big)$ 
lie on a common $\C$-circle.

\section{Background}

All material in this section is standard; see \cite{gol} for example unless otherwise indicated. 
	
\subsection{Complex Hyperbolic space}

Let $\C^{2,1}$ be the $3$-dimensional vector space over $\C$ equipped with the Hermitian form of signature $(2,1)$ given by 
$$
\langle\z,\w\rangle=\w^{\ast}H\z=\bar w_3 z_1+\bar w_2 z_2+\bar w_1 z_3,
$$
where $\z,\,\w$ are column vectors in $\C^3$ and the matrix of the Hermitian form is given by
\begin{center}
	$H =\left[ \begin{array}{cccc}
	0 & 0 & 1\\
	0 & 1 & 0 \\
	1 & 0 & 0\\
	\end{array}\right].$
\end{center}
If $\z \in \C^{2,1}$ then $\langle\z,\z\rangle$ is real. Thus we may consider the following subsets of $\C^{2,1}\setminus \{\bf 0\}:$
\begin{eqnarray*}
V_+ & = & \{\z\in\C^{2,1}:\langle\z,\z \rangle>0\}, \\
V_{-} & = & \{\z\in\C^{2,1}:\langle\z,\z \rangle<0\},\\
V_{0} & = & \{\z \in\C^{2,1}\setminus \{{\bf 0}\}:\langle\z,\z \rangle=0\}.
\end{eqnarray*}
A vector $\z$ in $\C^{2,1}$ is called positive, negative or null depending on whether $\z$ belongs to $V_+$, $V_-$ or  $V_0$ 
respectively. Let $\P:\C^{2,1}\setminus\{{\bf 0}\}\longrightarrow \C {\P}^{2}$ be the projection map onto the complex projective 
space. The complex hyperbolic space is defined to be $\c^2=\P (V_{-})$. The ideal boundary of complex hyperbolic space is 
$\partial \c^2=\P (V_{0})$.	 Let ${\rm U}(2,1)$ be the unitary group of above Hermitian form. The biholomorphic isometry group of 
$\c^2$ is the projective unitary group ${\rm PU}(2,1)$. In addition, the map $\z\longmapsto \bar\z$ that sends each entry of 
$\z$ to its complex conjugate $\bar \z$ is an anti-holomorphic isometry of $\c^2$. 
Any other anti-holomorphic isometry may be written as the projectivisation of the 
composition of this map and an element of ${\rm U}(2,1)$. All complex hyperbolic isometries are either holomorphic or 
anti-holomorphic.

We define Siegel domain model of complex hyperbolic space by taking the section defined by $z_3=1$ for the given Hermitian 
form. In other words, if $(z_1,z_2)$ is in $\C^2$, we define its standard lift to be $\z=(z_1,z_2,1)^t$ in $\C^{2,1}$. The Siegel
domain is the subset of $\C^2$ consisting of points whose standard lift lies in $V_-$. Specifically, the Siegel domain is:
$$
\c^2 =\{(z_1,z_2)\in \C^2 : \ 2\Re(z_1)+|z_2|^2<0\}.
$$ 
Now consider ${\bf q}_\infty=(1,0,0)^t$ be the column vector in $\C^{2,1}$. It is easy to see that 
${\bf q}_\infty\in V_0$. We define $\P({\bf q}_\infty)=\infty$, which lies in $\partial \c^2$. Any point in $\partial\c^2-\{\infty\}$ 
has a standard lift in $V_0$. That is,
$$
\partial \c^2-\{\infty\}=\P(V_{0}-\{{\bf q}_\infty\})=\{(z_1,z_2)\in \C^2 :2\Re(z_1)+|z_2|^2=0\}.
$$
The origin is the point $o\in \partial{\bf H}^2_{\mathbb C}={\mathbb P}(V_0)$ with $(z_1,z_2)=(0,0)$. 

Suppose $z_1,\,z_2,\,z_3$ are distinct points in $\partial{\bf H}^2_{\mathbb C}$ with standard lifts 
${\bf z}_1$, ${\bf z}_2$, ${\bf z}_3$ respectively, we define their Cartan angular invariant ${\mathbb A}(z_1,z_2,z_3)$ to be
$$
{\mathbb A}(z_1,z_2,z_3)
=\arg\bigl(-\langle{\bf z}_1,{\bf z}_2\rangle\langle{\bf z}_2,{\bf z}_3\rangle\langle{\bf z}_3,{\bf z}_1\rangle\bigr).
$$
We know that ${\mathbb A}(z_1,z_2,z_3)=\pm \frac{\pi}{2}$ (resp. ${\mathbb A}(z_1,z_2,z_3)=0$) if and only if $z_1,z_2,z_3$ lie on the same $\C$-circle (resp. $\R$-circle). 

\subsection{The Heisenberg group}

The set $\partial \c^2-\{\infty\}$ naturally carries the structure of the Heisenberg group ${\mathfrak N}$. 
Thus, the boundary of complex hyperbolic space is the one-point compactification of the Heisenberg group, 
which should be thought of as a generalisation of the well known fact that the boundary of the upper half space 
model of  $\r^3$ is the one point compactification of $\C$.

We recall that the Heisenberg group ${\mathfrak N}$ is $\C\times\R$ with the group law
$$
(\zeta_1,v_1)\cdot(\zeta_2,v_2)=\bigl(\zeta_1+\zeta_2,v_1+v_2+2\Im(\zeta_1\bar\zeta_2)\bigr).
$$
The identity element in the Heisenberg group is $(0,0)$, which we denote by $o$.
The map $\Pi_V:{\mathfrak N}\longrightarrow {\mathbb C}$ given by $\Pi_V:(\zeta,v)\longmapsto \zeta$ is called vertical
projection. It is a homomorphism.

Given $z=(\zeta,v)$ in the Heisenberg group, we define its standard lift to be
$$ 
{\bf z} = \begin{pmatrix} -|\zeta|^2 + iv \\ \sqrt{2}\zeta \\ 1 \end{pmatrix}\in V_0.
$$
Given $(\tau,t)$ in the Heisenberg group, there is a unique, upper triangular unipotent element of ${\rm U}(2,1)$ taking
the standard lift of $o=(0,0)$ to the standard lift of $(\tau,t)$, which is given by
$$
T_{(\tau,t)}=\begin{pmatrix} 1 & -\sqrt{2} \bar{\tau} & -|\tau|^2+it \\
0 & 1 & \sqrt{2}\tau \\
0 & 0 & 1 \end{pmatrix}.
$$
The map $(\tau,t)$ is a group homomorphism from ${\mathfrak N}$ to ${\rm U}(2,1)$. Thus, applying 
$T_{(\tau,t)}$ to the standard lift of points in ${\mathfrak N}$ is equivalent to ${\mathfrak N}$ acting on itself
by left translation; that is, the map
$(\zeta,v)\longmapsto (\tau,t)\cdot(\zeta,v)=\bigl(\tau+\zeta,t+v+2\Im(\tau\bar\zeta)\bigr)$.

We define Cygan metric on the Heisenberg group by 
$$
\rho_0\big((\zeta_1,v_1), (\zeta_2,v_2)\big)= \big||\zeta_1-\zeta_2|^2-iv_1+iv_2-2i\Im(\zeta_1\bar\zeta_2)  \big|^{\frac{1}{2}}.
$$
There is an easy way to compute the Cygan distance. If $z_1,\,z_2$ are two points in ${\mathfrak N}$ with standard lifts
${\bf z}_1,\,{\bf z}_2\in V_0$ respectively then
$$
\rho_0(z_1,z_2)=\bigl|\langle {\bf z}_1,{\bf z}_2\rangle\bigr|^{1/2}.
$$
We can define Cygan sphere $S_{(r,z_0)}$ of radius $r$ and centre $z_0 \in \partial \c^2$ by 
$$
S_{(r,z_0)}= \Bigl\{z \in \partial \c^2: \rho_0(z,z_0)=r\Bigr\}.$$

\subsection{Special subsets of complex hyperbolic space and its boundary}

There are two types of totally geodesic subspaces of $\c^2$ with real dimension 2. The first is the intersection with $\c^2$ of 
a complex line (copy of ${\mathbb C}{\mathbb P}^1$ inside ${\mathbb C}{\mathbb P}^2$). These are fixed by involutions that 
are holomorphic isometries of $\c^2$. Their intersection with $\partial\c^2$ are called chains or $\C$-circles. 
The second type of totally geodesic subspace is the intersection of $\c^2$ with a Lagrangian planes (copies of 
${\mathbb R}{\mathbb P}^2$ inside ${\mathbb C}{\mathbb P}^2$). These are fixed by involutions that are 
anti-holomorphic isometries of $\c^2$. Their intersection
with $\partial\c^2$ are called ${\mathbb R}$-circles. A ${\mathbb C}$-circle or an ${\mathbb R}$-circle is 
called infinite if it passes through $\infty$ and is called finite otherwise.

There are no totally geodesic real hypersurfaces in $\c^2$ and so it is necessary to make a choice for the hypersurfaces 
containing the sides of a fundamental domain. We make two choices in this paper. The first are Cygan spheres, which are
a particular cases of the boundary of bisectors, 
and the second are fans.  Both are foliated by complex lines and Lagrangian planes, and both are mapped to themselves by
the involution fixing each complex line in this foliation and by the involution fixing each Lagrangian plane in the foliation. 
In what follows, we will discuss the boundary of these hypersurfaces in $\partial\c^2={\mathfrak N}\cup\{\infty\}$.

\subsection{Isometric spheres}

In \cite{gol}, Goldman extended the definition of isometric spheres in real hyperbolic geometry to geometry of complex hyperbolic space. These are spheres in the Cygan metric and are a particular type of bisector.

Let $P$ be any element of $\PU(2,1)$ which does not fix $\infty$, then the isometric sphere corresponding to $P$ is given by 
$$
I_P= \Bigl\{z \in \partial \c^2: \bigl|\langle {\bf z}, q_\infty \rangle \bigr|= \bigl|\langle {\bf z}, P^{-1}(q_\infty)\rangle\bigr|\Bigr\}.
$$
Let $P$ be an element of $\PU(2,1)$ then $P^{-1}$ has the following form: 
\begin{equation}\label{eq P}
		P=\left(\begin{matrix} a & b & c \\ d & e &f  \\ g & h & j \end{matrix}\right),
		\qquad
		P^{-1}=\left(\begin{matrix} \bar j & \bar f & \bar c \\ \bar h & \bar e &\bar b  \\ \bar g & \bar d & \bar a \end{matrix}\right),
\end{equation}
Such a map does not fix $\infty$ if and only if $g\neq 0$. The isometric sphere of $P$; denoted by  $I_P$ is the sphere in 
Cygan metric with centre $P^{-1}(\infty)$ and radius $r_P=1/\sqrt{|g|}$. Similarly, the isometric sphere $I_{P^{-1}}$ is
the Cygan sphere with centre $P(\infty)$ and radius $r_{P^{-1}}=r_P=1/\sqrt{|g|}$. In Heisenberg coordinates, the centres of these
spheres are
$$
P^{-1}(\infty)= \left(\frac{\bar h}{\sqrt 2{\bar g}},-\Im(\frac{j}{g}) \right), \qquad
P(\infty)= \left(\frac{d}{\sqrt 2{g}},\Im(\frac{a}{g}) \right).
$$
We will need the following lemma.	

\begin{lemma}(Proposition $2.4$ of \cite{kam1}){\label{Cygan Sphere}}
Let $P$ be any element of $\PU(2,1)$ such that $P(\infty) \neq \infty$, then there exist $r_P$ such that for all $z \in \partial {{\bf H}^2_{\mathbb C}}\setminus\left\{\infty, P^{-1}(\infty)\right\}$ we have 
$$\rho_0\big(P(z),P(\infty)\big)=\frac{r_P^2}{\rho_0\big(z,P^{-1}(\infty)\big)}.
$$
\end{lemma}
Note that $P$ maps  $I_P$ to $I_{P^{-1}}$ and maps the component of 	$\overline {{\bf H}^2_{\mathbb C}} \setminus I_P$ containing $\infty$ to the component of $\overline {{\bf H}^2_{\mathbb C}} \setminus I_{P^{-1}}$ not containing $\infty$.

We use an involution $\iota$ swapping $o$ and $\infty$. It is defined as
\begin{equation}\label{eq-iota}
	\iota=\left(\begin{matrix} 0 & 0 & 1 \\ 0 & 1 & 0 \\ 1 & 0 & 0 \end{matrix}\right).
\end{equation}
For $(\zeta,v)\neq (0,0)$  the involution $\iota$ is given in Heisenberg coordinates as
$$
\iota(\zeta,v)=\left(\frac{\zeta}{-|\zeta|^2+iv},\,\frac{-v}{\bigl||\zeta|^2+iv\bigr|^2}\right).
$$
As a consequence of Lemma \ref{Cygan Sphere}, the involution $\iota$ maps the Cygan sphere with centre $o$ and radius $r$ 
to Cygan sphere with centre $o$ and radius $1/r$.
Also, conjugating $A$ by $\iota$ results in a matrix of the same form as $B$ but where the indices of $s_j$, $t_j$, $\theta_j$ 
are all $1$. Likewise, conjugating
$B$ by $\iota$ results in a matrix of the same form as $A$ but where the indices of $s_j$, $t_j$, $\theta_j$ are all 2.
 
We can give geographical coordinates on isometric spheres. For convenience when proving Lemma~\ref{lem-diam-Rcircle} we
modify the more usual coordinates by letting $\alpha$ vary in $[0,\pi/2]$. Specifically, we parametrise points on the Cygan sphere
with centre $o$ and radius $r>0$ by $s_{\alpha,\beta}$ where $(\alpha,\beta)\in[0,\pi/2]\times({\mathbb R}/2\pi{\mathbb Z})$, given by
$$
s_{\alpha,\beta}=\Bigl(r\sqrt{\sin(2\alpha)}e^{i\alpha+i\beta},r^2\cos(2\alpha)\Bigr).
$$
The point $s_{\alpha,\beta}$ has standard lift
$$
{\bf s}_{\alpha.\beta}=\left(\begin{matrix} r^2ie^{2i\alpha} \\ r\sqrt{2\sin(2\alpha)}e^{i\alpha+i\beta} \\ 1 \end{matrix}\right).
$$
Fixing $\alpha=\alpha_0$ gives a ${\mathbb C}$-circle. 
The union of the arcs where $\beta=\beta_0$ and $\beta=\beta_0+\pi$ gives an ${\mathbb R}$-circle.
The involutions fixing these ${\mathbb C}$-circles and ${\mathbb R}$-circles all map the sphere to itself. For each $\beta_0$
this involution preserves the Cygan distance between points of the sphere, but the only value of $\alpha_0$ where this is true
is $\alpha_0=\pi/4$, corresponding to the equator.
 
\subsection{Fans}

Fans are another class of surfaces in the Heisenberg group. They were introduced by Goldman and Parker in \cite{gp2}. 
An infinite fan is a Euclidean plane in ${\mathfrak N}$ whose image under vertical projection $\Pi_V$ is an affine line in 
${\mathbb C}$. Infinite fans are foliated by infinite ${\mathbb C}$-circles and infinite ${\mathbb R}$-circles. 
Given $ke^{i\phi}\in{\mathbb C}$, let $F^{(\infty)}_{ke^{i\phi}}$ be the fan whose image under vertical projection is the line
given by the equation $x\cos(\phi)+y\sin(\phi)=k$, where $z=x+iy$. We can write points of 
$$
F^{(\infty)}_{ke^{i\phi}}=\Bigl\{ f_{a,b}=\bigl((k+ia)e^{i\phi},\,b-2ka\bigr)\ :\ (a,b)\in{\mathbb R}^2\Bigr\}.
$$
The standard lift of $f_{a,b}$ is
$$
{\bf f}_{a,b}\left(\begin{matrix}
-a^2-k^2+ib-2ika \\ \sqrt{2}(k+ia)e^{i\phi} \\ 1 \end{matrix}\right).
$$
We can give this fan coordinates that resemble geographical coordinates. Fixing $a=a_0$ gives an infinite ${\mathbb C}$-circle
and fixing $b=b_0$ gives an infinite ${\mathbb R}$-circle. The involutions that fixing both of these ${\mathbb C}$-circles and the
${\mathbb R}$-circles are all Cygan isometries. 

We are also interested in fans that are the image of this one under the involution $\iota$. That is
$$
F^{(o)}_{ke^{i\phi}}=\left\{ \left(\frac{-(k+ia)(a^2+k^2+ib-2ika)e^{i\phi}}{(a^2+k^2)^2+(b-2ka)^2},\,
\frac{-b+2ka}{(a^2+k^2)^2+(b-2ka)^2}\right)\ :\ (a,b)\in{\mathbb R}^2\right\}.
$$

\subsection{A discreteness criterion}\label{sec-discreteness}

In order to show the group $\langle A,B\rangle$ is discrete and free we will consider its action on 
$\partial{\bf H}^2_{\mathbb C}={\mathfrak N}\cup\{\infty\}$ and we will use the Klein Combination theorem,
Proposition~\ref{KCT}.

The construction is the following. We will consider four topological spheres in ${\mathfrak N}\cup\{\infty\}$ called 
$S_A^+,\,S_A^-,\,S_B^+,\,S_B^-$. The complement of each of these spheres has two (open) components, which we call the interior
and the exterior. We assume that:
\begin{enumerate}
\item the interiors of $S_A^+,\,S_A^-,\,S_B^+,\,S_B^-$ are disjoint;
\item $A$ sends the exterior of $S_A^-$ onto the interior of $S_A^+$, and hence $A^{-1}$ sends the exterior of 
$S_A^+$ onto the interior of $S_A^-$;
\item $B$ sends the exterior of $S_B^-$ onto the interior of $S_B^+$, and hence $B^{-1}$ sends the exterior of 
$S_B^+$ onto the interior of $S_B^-$.
\end{enumerate}
Then the intersection of the exteriors, which we call $D$, is then a fundamental domain for $\langle A,B\rangle$. It is easy to see 
that if $W$ is a reduced word in $A^{\pm 1}$ and $B^{\pm 1}$ (that is all consecutive occurrences of $A^{\pm 1}A^{\mp 1}$
and $B^{\pm 1}B^{\mp 1}$ have been cancelled) then $W$ sends $D$ into the interiors of one of  $S_A^+,\,S_A^-,\,S_B^+,\,S_B^-$ 
corresponding to the last generator to be applied. In our constructions, the spheres will either be the Cygan spheres or 
they will be fans. 

Many different versions of this result have been used for complex hyperbolic isometries; see, for example, 
Proposition~6.3 of Parker \cite{par2} or the notion of a group being compressing, due to Wyss-Gallifent in \cite{wg} and used 
by Monaghan, Parker and Pratoussevitch in \cite{mpa}. 

We can consider how our construction relates to that of Parker and Will \cite{PW}. Suppose that the interiors of 
$S_A^+,\,S_A^-,\,S_B^+,\,S_B^-$ are disjoint, but that there are points $q_+=S_A^+\cap S_B^-$ and $q_-=S_A^-\cap S_B^+$.
The existence of such points implies that we have equality in the relevant expression of Theorem~\ref{thm-main0}.
If furthermore $B(q_+)=q_-$ and $A(q_-)=q_+$ then $q_+$ is a fixed point of $AB$ and $q_-$ is a fixed point of $BA$. The
existence of points with these properties in necessary of $AB$ is parabolic. 

\section{The Cygan diameter of a finite ${\mathbb R}$-circle}

In this section we prove Lemma~\ref{lem-diam-Rcircle}. Let $R$ be any finite ${\mathbb R}$-circle and let $\iota_R$ be the 
anti-holomorphic involution fixing $R$. Let $r$ be the radius of $R$, that is $r$ is the Cygan distance from 
$\iota_R(\infty)$ to any point of $R$. 
Applying a Cygan isometry (Heisenberg translation and rotation) if necessary, we may 
assume that $R$ has the following form
$$
R=\Bigl\{p_{\alpha,\epsilon}=\bigl(\epsilon r\sqrt{\sin(2\alpha)}e^{i\alpha},r^2\cos(2\alpha)\bigr)\ :\ 
\alpha\in[0,\pi/2],\ \epsilon=\pm1\Bigr\}.
$$
Perhaps the easiest way to see that this collection of points comprise an ${\mathbb R}$-circle is to consider the
following map $\iota_R$:
$$
\iota_R:\left(\begin{matrix} z_1 \\ z_2 \\ z_3 \end{matrix}\right) \longmapsto
\left(\begin{matrix} \bar{z}_3r^2 \\ -i\bar{z}_2 \\ \bar{z}_1/r^2 \end{matrix}\right).
$$
It is easy to check that $\langle \iota_R{\bf z},\iota_R{\bf w}\rangle=\overline{\langle{\bf z},{\bf w}\rangle}$ for any 
${\bf z}$, ${\bf w}$ in $\C^{2,1}$, and so $\iota_R$ is a complex hyperbolic isometry. Moreover, it is easy to
check that $\iota_R^2$ is the identity. Hence, by construction, the subset of $V_0$ projectively fixed by $\iota_R$
is an ${\mathbb R}$-circle. For $\alpha\in[0,\pi/2]$ and $\epsilon=\pm1$ consider
$$
{\bf p}_{\alpha.\epsilon}=\left(\begin{matrix} r^2ie^{2i\alpha} \\ \epsilon r\sqrt{2\sin(2\alpha)}e^{i\alpha} \\ 1 \end{matrix}\right)
\in V_0.
$$
Observe that $\iota_R:{\bf p}_{\alpha,\epsilon}\longmapsto (-ie^{-2i\alpha}){\bf p}_{\alpha,\epsilon}$ and so 
${\bf p}_{\alpha,\epsilon}$ is projectively fixed by $\iota_R$. We can express 
$p_{\alpha,\epsilon}={\mathbb P}{\bf p}_{\alpha,\epsilon}$ in Heisenberg coordinates as:
$$
p_{\alpha,\epsilon}=\Bigl(\epsilon r\sqrt{\sin(2\alpha)}e^{i\alpha},r^2\cos(2\alpha)\Bigr).
$$
This gives the result.

It is easy to check that ${\mathbb A}(p_{\alpha,\epsilon},o,\infty)=2\alpha-\pi/2$. Thus, there are two points $p_\alpha$ satisfying 
the condition ${\mathbb A}(p_{\alpha},\iota_R(\infty), \infty)=2\alpha-\pi/2$, namely $p_{\alpha,+1}$ and $p_{\alpha,-1}$. 
Given $\theta\in[0,\pi/2]$ and
$\eta=\pm 1$ the Cygan distance from $p_{\alpha,\epsilon}$ to $p_{\theta,\eta}$ is given by
\begin{eqnarray*}
\rho_0({p_{\alpha,\epsilon}, p_{\theta,\eta}})
& = & \Bigl| ir^2e^{2i\theta}-ir^2e^{-2i\alpha}
+2\eta\epsilon r^2\sqrt{\sin(2\alpha)\sin(2\theta)}e^{i\theta-i\alpha}\Bigr|^{1/2} \\
& = & \Bigl| -2r^2\sin(\theta+\alpha)e^{i\theta-i\alpha}
+2\eta\epsilon r^2\sqrt{\sin(2\alpha)\sin(2\theta)}e^{i\theta-i\alpha}\Bigr|^{1/2} \\
& = & 2^{1/2}r\Bigl( \sin(\alpha+\theta)-\eta\epsilon\sqrt{\sin(2\alpha)\sin(2\theta)}\Bigr)^{1/2} \\
& = & 2^{1/2}r\Bigl( \sin(\alpha)\cos(\theta)+\cos(\alpha)\sin(\theta)-\eta\epsilon
2\sqrt{\sin(\alpha)\cos(\alpha)\sin(\theta)\cos(\theta)}\Bigr)^{1/2} \\
& = & 2^{1/2}r\bigl|\sin^{1/2}(\alpha)\cos^{1/2}(\theta)-\eta\epsilon \cos^{1/2}(\alpha)\sin^{1/2}(\theta)\bigr|.
\end{eqnarray*}
We need to maximize this quantity. Taking positive square roots of all the trigonometric functions, we see that
this maximum arises when $\eta=-\epsilon$. We now use calculus to find the resulting maximum. Putting all
this together, gives a proof of Lemma~\ref{lem-diam-Rcircle}. 

\begin{lemma}\label{lem-trig2}
Given $\alpha\in[0,\pi/2]$ define $f_\alpha:[0,\pi/2]\longrightarrow {\mathbb R}$ by
$$
f_\alpha(\theta)=\sin^{1/2}(\alpha)\cos^{1/2}(\theta)+\cos^{1/2}(\alpha)\sin^{1/2}(\theta).
$$
Then 
$$
f_\alpha(\theta)\le \bigl(\sin^{2/3}(\alpha)+\cos^{2/3}(\alpha)\bigr)^{3/4}.
$$
\end{lemma}

\begin{proof}
First observe
\begin{eqnarray*}
f_\alpha(0) & = & \sin^{1/2}(\alpha)\le \bigl(\sin^{2/3}(\alpha)+\cos^{2/3}(\alpha)\bigr)^{3/4}, \\
f_\alpha(\pi/2) & = & \cos^{1/2}(\alpha)\le \bigl(\sin^{2/3}(\alpha)+\cos^{2/3}(\alpha)\bigr)^{3/4}.
\end{eqnarray*}
In the first case, there is equality if and only if $\alpha=\pi/2$ and in the second case if and only if
$\alpha=0$.

Now suppose $\theta\in(0,\pi/2)$. Differentiating with respect to $\theta$ we have
$$
f'_\alpha(\theta)=\frac{-\sin(\theta)\sin^{1/2}(\alpha)}{2\cos^{1/2}(\theta)}
+\frac{\cos(\theta)\cos^{1/2}(\alpha)}{2\sin^{1/2}(\theta)}.
$$
Therefore, if $\theta_0$ is a value of $\theta$ for which $f_\alpha'(\theta_0)=0$ we have
$$
\sin^{3/2}(\theta_0)\sin^{1/2}(\alpha)=\cos^{3/2}(\theta_0)\cos^{1/2}(\alpha).
$$
In other words, $\cos^{1/2}(\theta_0)=k\sin^{1/6}(\alpha)$ and $\sin^{1/2}(\theta_0)=k\cos^{1/6}(\alpha)$ for
some constant $k$, depending on $\alpha$. Using $1=\cos^2(\theta_0)+\sin^2(\theta_0)$ we have:
$$
1=k^4\bigl(\sin^{2/3}(\alpha)+\cos^{2/3}(\alpha)\bigr).
$$
Hence:
\begin{eqnarray*}
\cos^{1/2}(\theta_0) & = & \frac{\sin^{1/6}(\alpha)}{\bigl(\sin^{2/3}(\alpha)+\cos^{2/3}(\alpha)\bigr)^{1/4}}, \\
\sin^{1/2}(\theta_0) & = & \frac{\cos^{1/6}(\alpha)}{\bigl(\sin^{2/3}(\alpha)+\cos^{2/3}(\alpha)\bigr)^{1/4}}.
\end{eqnarray*}
Therefore
\begin{eqnarray*}
f_\alpha(\theta_0) & = & \sin^{1/2}(\alpha)\cos^{1/2}(\theta_0)+\cos^{1/2}(\alpha)\sin^{1/2}(\theta) \\
& = &  \frac{\sin^{2/3}(\alpha)}{\bigl(\sin^{2/3}(\alpha)+\cos^{2/3}(\alpha)\bigr)^{1/4}}
+\frac{\cos^{2/3}(\alpha)}{\bigl(\sin^{2/3}(\alpha)+\cos^{2/3}(\alpha)\bigr)^{1/4}} \\
& = & \bigl(\sin^{2/3}(\alpha)+\cos^{2/3}(\alpha)\bigr)^{3/4}.
\end{eqnarray*}
\end{proof}

\medskip

\section{Fundamental domain bounded by Cygan spheres} \label{sec-fd-bis}

First consider $B$ as given in equation \eqref{eq-A-B}. 
The isometric sphere of $B$ is a Cygan sphere of radius $r_B=1/|s_2^2+it_2|^{1/2}$ with centre $B^{-1}(\infty)$. Likewise,
the isometric sphere of $B^{-1}$ is a Cygan sphere of the same radius with centre $B(\infty)$. By construction $o$, the 
fixed point of $B$, lies on both these spheres and they are both tangent at this point. In this section, we use the
result on the diameters of an finite ${\mathbb R}$-circle to find the smallest Cygan sphere centred at $o$ containing both these 
isometric spheres.

Now consider $\iota A\iota$ and do a similar thing. The isometric sphere of $\iota A\iota$ is a Cygan sphere
of radius $r_A=1/|s_1^2+it_1|^{1/2}$ with centre $\iota A^{-1}\iota(\infty)=\iota A^{-1}(o)$ and the isometric sphere
of $\iota A^{-1}\iota$ is a Cygan sphere of the same radius with centre $\iota A(o)$. Again, we want to find the smallest
Cygan sphere centred at $o$ containing both these isometric spheres. Now apply $\iota$ to find the largest Cygan sphere 
with the images
of these isometric spheres in its exterior. If this sphere has large enough radius, these two spheres will be disjoint from
the isometric spheres of $B$ and $B^{-1}$. Thus the interiors of these four spheres will be disjoint, and the result will
follow as in Section~\ref{sec-discreteness}.

\subsection{A Cygan ball containing the isometric spheres of $B$ and $B^{-1}$}

\begin{proposition}\label{prop-iso-diam}
	The isometric spheres of $B$ and $B^{-1}$ are contained in the Cygan ball ${\mathcal B}$ 
	with centre $o$ and radius $d_B$ where
	$$
	d_B=\frac{2^{1/4}}{|s_2^2+it_2|^{1/2}}
	\left(\left(1-\frac{t_2}{|s_2^2+it_2|}\right)^{1/3}+\left(1+\frac{t_2}{|s_2^2+it_2|}\right)^{1/3}\right)^{3/4}.
	$$
	Furthermore, let $p_+$ and $p_-$ be the points of these two isometric spheres on the boundary of
	${\mathcal B}$. Then $B$ sends $p_+$ to $p_-$.
\end{proposition}

\begin{proof} 
We want to find the largest Cygan distance from $o$ to another point on the isometric sphere of $B$, respectively $B^{-1}$.
Since Cygan spheres are strictly convex, each of these points $p_+$, respectively $p_-$, must be unique. 
Let $R_+$, respectively $R_-$, 
be the meridian of the isometric sphere of $B$, respectively $B^{-1}$, passing through $o$. We claim that $p_+$ lies on $R_+$.
Observe that there is an anti-holomorphic involution $\iota_{R_+}$ whose fixed point set is $R_+$ which maps the
isometric sphere of $B$ to itself isometrically. Thus if $p_+$ does not lie on $R_+$ then $\iota_{R_+}(p_+)$ is a point of the
isometric sphere of $B$ different from $p_+$ and the same distance from $o$. This contradicts uniqueness of $p_+$. 
Thus $p_+$ lies on $R_+$ and similarly $p_-$ lies on $R_-$.

This means we can use Lemma~\ref{lem-diam-Rcircle} to find the distances betwen $o$ and $p_\pm$. 
The Cygan isometric sphere of $B$ has radius $r_B=1/|s_2^2+it_2|^{1/2}$ 
and centre
$$
B^{-1}(\infty)=\left( \frac{s_2 e^{i\theta_2}}{s_2^2+it_2},\  \frac{t_2}{|s_2^2+it_2|^2}\right).
$$
Simillarly, The Cygan isometric sphere of $B^{-1}$ has radius $r_{B^{-1}}=r_B =1/|s_2^2+it_2|^{1/2}$ and centre
$$
B(\infty)=\left( \frac{s_2 e^{i\theta_2}}{-s_2^2+it_2},\  \frac{-t_2}{|s_2^2+it_2|^2}\right).
$$ 
By construction, $o$ lies on both these spheres. It is easy to see that
$$
{\mathbb A}(o,B^{-1}(\infty),\infty)=\arg(s_2^2+it_2),\quad {\mathbb A}(o,B(\infty),\infty)=\arg(s_2^2-it_2).
$$
Define $\alpha$ by $2\alpha-\pi/2={\mathbb A}(o,B^{-1}(\infty),\infty)$. Then 
${\mathbb A}(o,B(\infty),\infty)=-2\alpha+\pi/2$. This means that
$$
\cos(2\alpha)=\frac{t_2}{|s_2^2+it_2|}.
$$
	Hence
	\begin{eqnarray*}
	\sin^2(\alpha) & = & \frac{1-\cos(2\alpha)}{2} = \frac{1}{2}\left(1-\frac{t_2}{|s_2^2+it_2|}\right), \\
	\cos^2(\alpha) & = & \frac{1+\cos(2\alpha)}{2} = \frac{1}{2}\left(1+\frac{t_2}{|s_2^2+it_2|}\right).
	\end{eqnarray*}
	Therefore the isometric sphere of $B$ is contained in the ball with centre $o$ and radius
	\begin{eqnarray*}
	d_B &  = & 2^{1/2}r_B\,\bigl(\cos^{2/3}(\alpha)+\sin^{2/3}(\alpha)\bigr)^{3/4} \\
	& = & \frac{2^{1/4}}{|s_2^2+it_2|^{1/2}}
	\left(\left(1-\frac{t_2}{|s_2^2+it_2|}\right)^{1/3}+\left(1+\frac{t_2}{|s_2^2+it_2|}\right)^{1/3}\right)^{3/4}.
	\end{eqnarray*}
The same is true for the isometric sphere of $B^{-1}$. This completes the proof of the first part.

To prove the second part, we observe that there is an infinite ${\mathbb R}$-circle $R_0$ passing through $o$ so that
the inversion $\iota_{R_0}$ fixing $R_0$ interchanges $B^{-1}(\infty)$ and $B(\infty)$. Hence this inversion also interchanges
the isometric spheres of $B$ and $B^{-1}$. Since it fixes the origin, it must also interchange $p_+$ and $p_-$. Assume
that $s_2\neq 0$. Then $R_+$ and $R_-$ are the unique meridians of these two isometric spheres passing through $o$. Hence
$B=\iota_{R_0}\iota_{R_+}=\iota_{R_-}\iota_{R_0}$. Thus by construction
$$
B(p_+)=\iota_{R_0}\iota_{R_+}(p_+)=\iota_{R_0}(p_+)=p_-.
$$
Finally, when $s_2=0$ then assume without loss of generality that $t_2>0$ (otherwise replace $B$ with $B^{-1}$). In this case,
$p_+$, $o$ are the north and south poles of the isometric sphere of $B$ and $p_-$, $o$ are the south and north poles of the 
isometric sphere of $B^{-1}$. It is clear that $B$ sends $p_+$ to $p_-$.
\end{proof}

\medskip

\subsection{Proof of Theorem \ref{thm-main0}$(1)$}

To obtain condition $(1)$ of Theorem \ref{thm-main0}, we do a similar thing  for $A$ using a Cygan metric where $o$ is the infinite point. We use the involution $\iota$
given in \eqref{eq-iota}. Recall that $\iota$ maps the Cygan sphere centred at $o$ of radius $d$ to the Cygan sphere centred at $o$ with radius $1/d$.
\begin{proof}
 By a  similar argument to Proposition~\ref{prop-iso-diam} shows that the isometric spheres of $\iota A\iota$ and 
$\iota A^{-1}\iota$ are contained in the Cygan ball with centre $o$ and radius $d_A$ where
	$$
	d_A=\frac{2^{1/4}}{|s_1^2+it_1|^{1/2}}
	\left(\left(1-\frac{t_1}{|s_1^2+it_1|}\right)^{1/3}+\left(1+\frac{t_1}{|s_1^2+it_1|}\right)^{1/3}\right)^{3/4}.
	$$

Therefore, the image under $\iota$ of these spheres is contained in the exterior of a Cygan sphere with centre $o$ and radius
$1/d_A$. Hence, if $d_B\le 1/d_A$ then we can use the Klein combination theorem, Proposition~\ref{KCT},
to conclude that $A$ and $B$ freely generate $\langle A,\,B\rangle$. The condition $d_B\le 1/d_A$ is equivalent to
$$
1\ge d_Ad_B = \frac{2^{1/2}}{|s_2^2+it_2|^{1/2}|s_1^2+it_1|^{1/2}}
\left(\left(1-\frac{t_2}{|s_2^2+it_2|}\right)^{1/3}+\left(1+\frac{t_2}{|s_2^2+it_2|}\right)^{1/3}\right)^{3/4}\\$$

$$\times \left(\left(1-\frac{t_1}{|s_1^2+it_1|}\right)^{1/3}+\left(1+\frac{t_1}{|s_1^2+it_1|}\right)^{1/3}\right)^{3/4}.
$$
Multiplying through by $|s_1^2+it_1|^{1/2}|s_2^2+it_2|^{1/2}$ we obtain condition (1) of Theorem~\ref{thm-main0}.
\end{proof}

\medskip

Note, that if we simply wanted to obtain condition (1') of Theorem~\ref{thm-main1}, we could use the triangle inequality to say that
any two points on a Cygan sphere of radius $r$ are a distance at most $2r$ apart. Therefore all points on the isometric spheres 
of $B$ and $B^{-1}$ lie within a Cygan distance $2/|s_2^2+it_2|^{1/2}$. Similarly, all points on the isometric spheres of 
$\iota A\iota$ and $\iota A^{-1}\iota$ lie within a a Cygan distance $2/|s_1^2+it_1|^{1/2}$. Arguing as above, but with these weaker bounds, we obtain the condition.
$$
1 \ge \frac{2}{|s_1^2+it_1|^{1/2}}\,\frac{2}{|s_2^2+it_2|^{1/2}}.
$$
Multiplying through by $|s_1^2+it_1|^{1/2}|s_2^2+it_2|^{1/2}$ we obtain condition (1') of Theorem~\ref{thm-main1}.

\subsection{Criteria for equality in Theorem \ref{thm-main0}$(1)$}

We now briefly discuss what happens when we have equality in the criterion of Theorem \ref{thm-main0}$(1)$. By construction,
there are points $p_+$ and $p_-=B(p_+)$ on the isometric spheres of $B$ and $B^{-1}$ so that both $p_+$ and $p_-$ are
a distance $d_B$ from $o$. Similarly, there are points $q_+$ and $q_-$ on the
$\iota$ images of the isometric spheres of $\iota A\iota$ and $\iota A^{-1}\iota$ so that $A$ sends $q_+$ to $q_-=A(q_+)$ and
$q_+$ and $q_-$ are a distance $1/d_A=d_B$ from $o$. If we have $q_-=p_+$ and $q_+=p_-$ then $p_+$ is a fixed point
of $AB$ and $p_-$ is a fixed point of $BA$. We claim this only happens when either (a) $s_1=s_2=0$ and $t_1t_2=4$ or
(b) $t_1=t_2=0$, $\theta_1=\theta_2$ and $s_1s_2=4$. In the first case $AB$ is screw parabolic with angle $\pi$ and in the second case $AB$ is unipotent.

Write $-s_2^2+it_2=ir_2^2e^{2i\alpha_2}$. 
Define $\phi_2$ by
$$
\cos(\phi_2)=\frac{\sin^{1/3}(\alpha_2)}{\bigl(\sin^{2/3}(\alpha_2)+\cos^{2/3}(\alpha_2)\bigr)^{1/2}},\quad 
\sin(\phi_2)=\frac{\cos^{1/3}(\alpha_2)}{\bigl(\sin^{2/3}(\alpha_2)+\cos^{2/3}(\alpha_2)\bigr)^{1/2}}
$$
and $d_B$ by
$$
d_B=r_2^{-1}\bigl(\sqrt{2\cos(\phi_2)\sin(\alpha_2)}+\sqrt{2\sin(\phi_2)\cos(\alpha_2)}\bigr).
$$

Let $T_+$, respectively $T_-$, be the Heisenberg translation taking $o$ to $B^{-1}(\infty)$, respectively $B(\infty)$. 
Then
$$
T_+^{-1}(o)=\left(\begin{matrix} 
	-ir_2^{-2}e^{-2i\alpha_2} \\ i\sqrt{2\sin(2\alpha_2)}r_2^{-1}e^{2i\alpha_2+i\theta_2} \\
	1 \end{matrix}\right), \quad
T_-^{-1}(o)=\left(\begin{matrix} 
	ir_2^{-2}e^{2i\alpha_2} \\ i\sqrt{2\sin(2\alpha_2)}r_2^{-1}e^{-2i\alpha_2+i\theta_2} \\
	1 \end{matrix}\right).
$$
The points on the ${\mathbb R}$-circle furthest from these two points are
$$
\left(\begin{matrix} 
	-ir_2^{-2}e^{-2i\phi_2} \\ -i\sqrt{2\sin(2\phi_2)}r_2^{-1}e^{-i\phi_2+3i\alpha_2+i\theta_2} \\ 
	1 \end{matrix}\right), \quad
\left(\begin{matrix} 
	ir_2^{-2}e^{2i\phi_2} \\ i\sqrt{2\sin(2\phi_2)}r_2^{-1}e^{i\phi_2-3i\alpha_2+i\theta_2} \\ 
	1 \end{matrix}\right).
$$
The images of these two points under $T_+$ and $T_-$ are $p_+$ and $p_-$. As vectors in standard form these are:
$$
{\bf p}_+=\left(\begin{matrix}
	-d_B^2 e^{-i\phi_2+i\alpha_2} \\
	d_B\sqrt{2\cos(\phi_2-\alpha_2)}e^{-2i\phi_2+2i\alpha_2+i\theta_2} \\
	1 \end{matrix}\right), \quad
{\bf p}_-=\left(\begin{matrix}
	-d_B^2 e^{i\phi_2-i\alpha_2} \\
	-d_B\sqrt{2\cos(\phi_2-\alpha_2)}e^{2i\phi_2-2i\alpha_2+i\theta_2} \\
	1 \end{matrix}\right).
$$

A similar caculation shows that
$$
B:{\bf p}_+ \longmapsto e^{-2i\phi_2+2i\alpha_2}{\bf p}_-.
$$
Thus, $B$ sends the vector ${\bf p}_+$ to the vector ${\bf p}_-$ with a non-trivial multiplier.
Applying the involution $\iota$ and changing all the indices from $2$ to $1$ gives
$$
\iota{\bf q}_+=\left(\begin{matrix}
	-d_A^2 e^{-i\phi_1+i\alpha_1} \\
	d_A\sqrt{2\cos(\phi_1-\alpha_1)}e^{-2i\phi_1+2i\alpha_1+i\theta_1} \\
	1 \end{matrix}\right), \quad
\iota{\bf q}_-=\left(\begin{matrix}
	-d_A^2 e^{i\phi_1-i\alpha_1} \\
	-d_A\sqrt{2\cos(\phi_1-\alpha_1)}e^{2i\phi_1-2i\alpha_1+i\theta_1} \\
	1 \end{matrix}\right),
$$
where 
\begin{eqnarray*}
	\cos(\phi_1) & = & \frac{\sin^{1/3}(\alpha_1)}{\bigl(\sin^{2/3}(\alpha_1)+\cos^{2/3}(\alpha_1)\bigr)^{1/2}},\quad 
	\sin(\phi_1) \ =\ \frac{\cos^{1/3}(\alpha_1)}{\bigl(\sin^{2/3}(\alpha_1)+\cos^{2/3}(\alpha_1)\bigr)^{1/2}}, \\
	d_A & = & r_1^{-1}\bigl(\sqrt{2\cos(\phi_1)\sin(\alpha_1)}+\sqrt{2\sin(\phi_1)\cos(\alpha_1)}\bigr).
\end{eqnarray*}
Applying $\iota$ we see that
\begin{eqnarray*}
	{\bf q}_+ & = & \left(\begin{matrix}
		1 \\
		d_A\sqrt{2\cos(\phi_1-\alpha_1)}e^{-2i\phi_1+2i\alpha_1+i\theta_1} \\
		-d_A^2 e^{-i\phi_1+i\alpha_1}  \end{matrix}\right) 
	\sim \left(\begin{matrix}
		-d_A^{-2} e^{i\phi_1-i\alpha_1} \\
		-d_A^{-1} \sqrt{2\cos(\phi_1-\alpha_1)}e^{-i\phi_1+i\alpha_1+i\theta_1} \\
		1 \end{matrix}\right), \\
	\\
	{\bf q}_- & = &\left(\begin{matrix}
		1 \\
		-d_A\sqrt{2\cos(\phi_1-\alpha_1)}e^{2i\phi_1-2i\alpha_1+i\theta_1} \\
		-d_A^2 e^{i\phi_1-i\alpha_1}  \end{matrix}\right) 
	\sim \left(\begin{matrix}
		-d_A^{-2} e^{-i\phi_1+i\alpha_1} \\
		d_A^{-1} \sqrt{2\cos(\phi_1-\alpha_1)}e^{i\phi_1-i\alpha_1+i\theta_1} \\
		1 \end{matrix}\right).
\end{eqnarray*}
Thus we have $$
A:{\bf q}_+ \longmapsto {\bf q}_-.
$$
Hence, $A$ sends the vector ${\bf q}_+$ to the vector ${\bf q}_-$ with a multiplier $1$.
Now, ${\bf p}_+ = {\bf q}_-$ and ${\bf p}_- = {\bf q}_+$ if and only if either (a) $s_1=s_2=0$ and $t_1t_2=4$ or
(b) $t_1=t_2=0$, $\theta_1=\theta_2$ and $s_1s_2=4$. In the first case $AB$ is a screw parabolic with angle $\pi$ and having fixed point ${ p}_+$. This is because the non-trivial multiplier is an eigenvalue of $AB$ associated to its fixed point. Also, in the second case $AB$ is an unipotent element.

\section{Domain bounded by fans}\label{sec-fd-fans}

Consider the infinite fan $F^{(\infty)}_{ke^{i\theta_1}}$ where $\theta_1$ is the angle associated to the Heisenberg translation
$A$ given by \eqref{eq-A-B} and $k$ is any real number.  We claim that $A$ sends $F_{ke^{i\theta_1}}^{(\infty)}$ to 
$F_{(k+s_1)e^{i\theta_1}}^{(\infty)}$. This is most easily seen by considering the standard lifts of points on the fans.
\begin{eqnarray*}
\lefteqn{
\left(\begin{matrix} 1 & -\sqrt{2}s_1e^{-i\theta_1} & -s_1^2+it_1 \\ 0 & 1 & \sqrt{2}s_1e^{i\theta_1} \\ 0 & 0 & 1 \end{matrix}\right)
\left(\begin{matrix} -k^2-a^2+ib-2ika \\ \sqrt{2}(k+ia)e^{i\theta_1} \\ 1 \end{matrix}\right) }\\
& = & \left(\begin{matrix} -(k+s_1)^2-a^2+ib-2i(k+s_1)a+it_1 \\ \sqrt{2}(k+s_1+ia)e^{i\theta_1} \\ 1 \end{matrix}\right).
\end{eqnarray*}
The images of $F^{(\infty)}_{ke^{i\theta_1}}$ and $F^{(\infty)}_{(k+s_1)e^{i\theta_1}}$ under vertical projection $\Pi_V$ are the lines
$$
x\cos(\theta_1)+y\sin(\theta_1)=k,\quad x\cos(\theta_1)+y\sin(\theta_1)=k+s_1.
$$

The slab bounded by the fans $F^{(\infty)}_{ke^{i\theta_1}}$ and $F^{(\infty)}_{(k+s_1)e^{i\theta_1}}$ is a fundamental region for 
$\langle A\rangle$. The image of the slab under vertical projection $\Pi_V$ is the strip 
$$
S_A(k)=\left\{(x,y)\ :\ 
k\le  x\cos(\theta_1)+y\sin(\theta_1) \le k+s_1 \right\}.
$$

\begin{figure}[htbp]
  \includegraphics[width=0.7\textwidth]{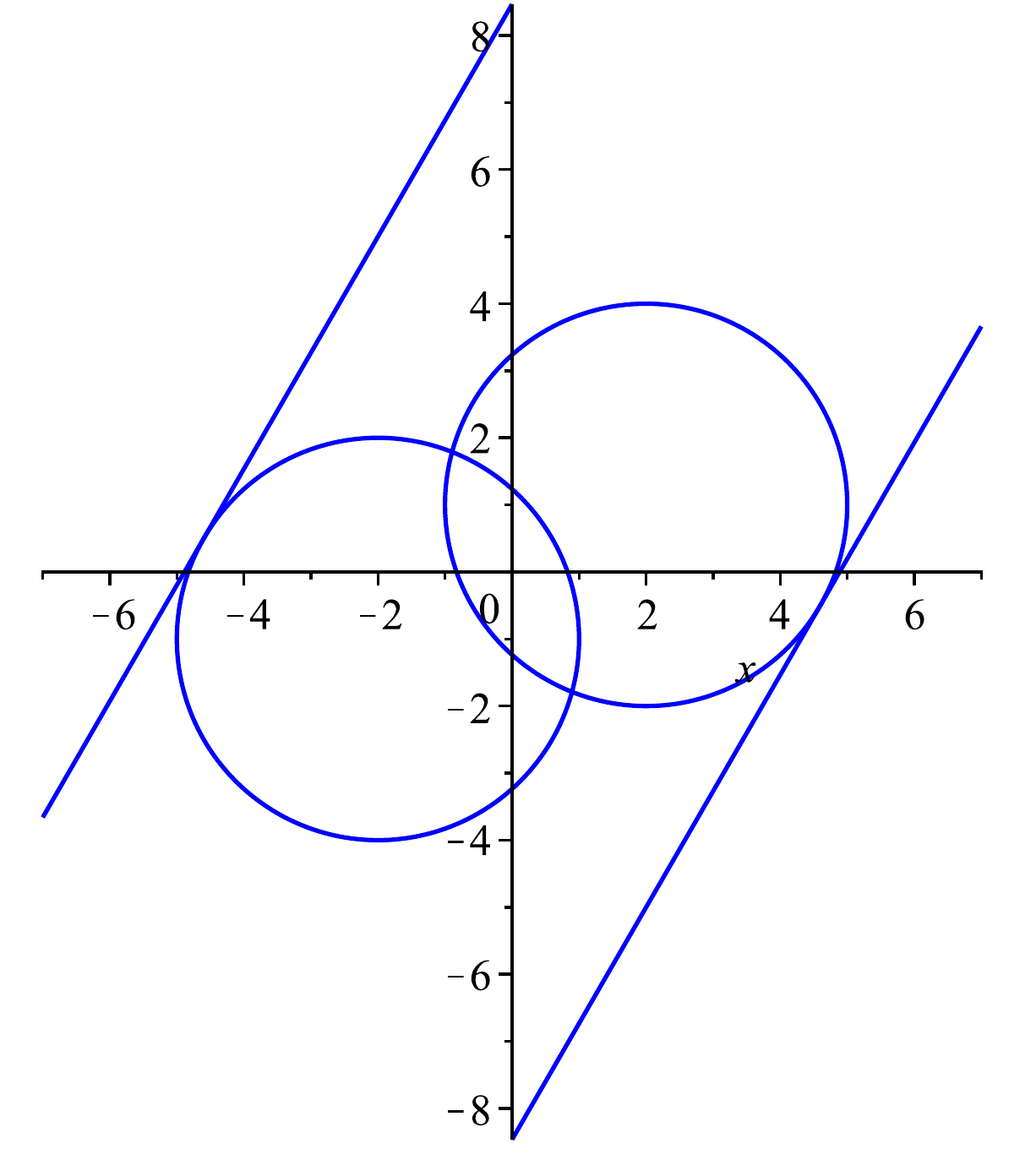}
  \caption{The vertical projection of the isometric spheres of $B$ and $B^{-1}$ and strip which is the vertical projection of
a fundamental domain for $\langle A\rangle$ bounded by two infinite fans. We need the projections of both isometric spheres
to lie in the strip.}\label{fig:1}
\end{figure}

Therefore, if we can find a fundamental domain $D_B$ for $\langle B\rangle$ containing the complement of this slab then we can
apply the Klein combination theorem to conclude that $\langle A,B\rangle$ is freely generated by $A$ and $B$. In particular,
this is true if the vertical projection of $\partial D_B$ is contained in $S_A$ and the vertical projection of $D_B$ contains 
the complement of $S_A$. To see this, consider a point in the complement of $S_A$. By construction, its pre-image under 
vertical projection contains at least one point of $D_B$ and no points in its boundary. Since it is connected and path connected
(the fibres of vertical projection are copies of the real line) all points in this preimage are contained in $D_B$ as required.

In what follows, we give two different fundamental domains for $\langle B\rangle$.

\subsection{Proof of Theorem \ref{thm-main0} parts (2) and (3)}

\begin{proposition}\label{prop-fan-bisector}
Let $(s_1e^{i\theta_1},t_1)$ and $(s_2e^{i\theta_2},t_2)$ be elements of the Heisenberg group with $s_1$ and $s_2$ both non-zero. 
Replacing one of them by its inverse if necessary, we suppose $-\pi/2\le (\theta_1-\theta_2)\le \pi/2$. 
Let $A$ and $B$ given by \eqref{eq-A-B} be the 
associated Heisenberg translations fixing $\infty$ and $o$ respectively. Suppose that
$$
s_1\,|s_2^2+it_2|^{1/2} \ge 
\frac{2~s_2}{|s_2^2+it_2|^{1/2}}\,\cos(\theta_1-\theta_2)+2. 
$$
Then the vertical projection of the isometric spheres of $B$ and $B^{-1}$ are contained in the strip $S_A$.
\end{proposition}

\begin{proof} (Proposition~\ref{prop-fan-bisector}).
The vertical projection of the isometric sphere of $B$ is a circle with radius $r_2=1/|s_2^2+it_2|^{1/2}$ and 
centre, 
$$
\frac{s_2e^{i\theta_2}}{s_2^2+it_2}=\frac{s_2e^{i\theta_2}(s_2^2-it_2)}{|s_2^2+it_2|^2}.
$$
Similarly, The vertical projection of the isometric sphere of $B^{-1}$ is a circle with radius 
$r_2=1/|s_2^2-it_2|^{1/2}$ and 
centre, 
$$ \frac{s_2e^{i\theta_2}}{-s_2^2+it_2}=\frac{-s_2e^{i\theta_2}(s_2^2+it_2)}{|s_2^2+it_2|^2}.
$$
Now, for a point $(x,y)$ on the vertical projection of the isometric sphere of $B$ we have
\begin{eqnarray*}
x & = & \frac{s_2(s_2^2\cos(\theta_2)+t_2\sin(\theta_2))}{|s_2^2+it_2|^2}+\frac{\cos(\phi)}{|s_2^2+it_2|^{1/2}},\\
y & = & \frac{s_2(s_2^2\sin(\theta_2)-t_2\cos(\theta_2))}{|s_2^2+it_2|^2}+\frac{\sin(\phi)}{|s_2^2+it_2|^{1/2}}.
\end{eqnarray*}
Therefore
$$
x\cos(\theta_1)+y\sin(\theta_1) 
= \frac{s_2^3\cos(\theta_1-\theta_2)-s_2t_2\sin(\theta_1-\theta_2)}{|s_2^2+it_2|^2}+\frac{\cos(\phi-\theta_1)}{|s_2^2+it_2|^{1/2}}.
$$
Similarly, for a point on the vertical projection of the isometric sphere of $B^{-1}$ we have
$$
x\cos(\theta_1)+y\sin(\theta_1) 
= \frac{-s_2^3\cos(\theta_1-\theta_2)-s_2t_2\sin(\theta_1-\theta_2)}{|s_2^2+it_2|^2}+\frac{\cos(\phi-\theta_1)}{|s_2^2+it_2|^{1/2}}.
$$
As $\cos(\theta_1-\theta_2)\ge 0$, we can see that for all points $(x,y)$ in the vertical projections of the isometric spheres
of $B$ and $B^{-1}$ we have
\begin{eqnarray*}
\lefteqn{  
\frac{-s_2^3\cos(\theta_1-\theta_2)}{|s_2^2+it_2|^2}-\frac{1}{|s_2^2+it_2|^{1/2}}
-\frac{s_2t_2\sin(\theta_1-\theta_2)}{|s_2^2+it_2|^2} } \\
& \le & x\cos(\theta_1)+y\sin(\theta_1) \\
& \le & 
 \frac{s_2^3\cos(\theta_1-\theta_2)}{|s_2^2+it_2|^2}
 +\frac{1}{|s_2^2+it_2|^{1/2}}
 -\frac{s_2t_2\sin(\theta_1-\theta_2)}{|s_2^2+it_2|^2}.
 \end{eqnarray*}
Hence these vertical projections lie in a strip of the form $S_A(k)$ provided
$$
 \frac{2s_2^3\cos(\theta_1-\theta_2)}{|s_2^2+it_2|^2} +\frac{2}{|s_2^2+it_2|^{1/2}}\le s_1.
 $$
Multiplying through by $|s_2^2+it_2|^{1/2}$, we get following condition
$$
s_1\,|s_2^2+it_2|^{1/2} \ge 
\frac{2~s_2}{|s_2^2+it_2|^{1/2}}\,\cos(\theta_1-\theta_2)+2. 
$$
\end{proof}

\medskip

The case of Theorem~\ref{thm-main0} arising from condition (2) follows immediately by the above argument.
Swapping the roles of $A$ and $B$ also gives the case arising from condition (3) of Theorem~\ref{thm-main0}.
\medskip

\begin{proposition}\label{prop-fan-bisector-2}
Let $(s_1e^{i\theta_1},t_1)$ and $(0,t_2)$ be elements of the Heisenberg group.
Let $A$ and $B$ given by \eqref{eq-A-B} be the 
associated Heisenberg translations fixing $\infty$ and $o$ respectively with $s_2=0$. Suppose that
$$
s_1\,|t_2|^{1/2} \ge 2.
$$
Then the vertical projection of the isometric spheres of $B$ and $B^{-1}$ are contained in the strip $S_A$.
\end{proposition}

\begin{proof}
This follows from the proof of Proposition~\ref{prop-fan-bisector} but is much simpler. 
In this case, the vertical projections of the isometric spheres of both $B$ and $B^{-1}$ are centred at the origin. 

If $(x,y)$ is a point on the vertical projection of either of these two isometric spheres are given respectively by
$$ 
x=\frac{\cos(\phi)}{|t_2|^{1/2}},\quad y=\frac{\sin(\phi)}{|t_2|^{1/2}}.
$$
Hence
$$
\frac{-1}{|t_2|^{1/2}}\le x\cos(\theta_1)+y\sin(\theta_1)=\frac{\cos(\phi-\theta_1)}{|t_2|^{1/2}}\le \frac{1}{|t_2|^{1/2}}.
$$
For these two Cygan spheres to lie in a strip of the form $S_A(k)$ we must have $2/|t_2|^{1/2}\le s_1$. This gives the result.
\end{proof}

\medskip

\subsection{Criteria for equality in Theorem \ref{thm-main0}(2)}

When we have equality in Proposition~\ref{prop-fan-bisector}, the bisectors for $B$ and the fans for $A$ are tangent at 
the two points
\begin{eqnarray*}
{\bf q}_+ & = & \left(\begin{matrix}
-1/|s_2^2+it_2|-2s_2(s_2^2+it_2)e^{i\theta_1-i\theta_2}/|s_2^2+it_2|^{5/2}-(s_2^2-it_2)/|s_2^2+it_2|^2 \\
\sqrt{2}s_2(s_2^2-it_2)e^{i\theta_2}/|s_2^2+it_2|^2+\sqrt{2}e^{i\theta_1}/|s_2^2+it_2|^{1/2} \\
1 \end{matrix}\right), \\
{\bf q}_- & = & \left(\begin{matrix}
-1/|s_2^2+it_2|-2s_2(s_2^2-it_2)e^{i\theta_1-i\theta_2}/|s_2^2+it_2|^{5/2}-(s_2^2+it_2)/|s_2^2+it_2|^2 \\
-\sqrt{2}s_2(s_2^2+it_2)e^{i\theta_2}/|s_2^2+it_2|^2-\sqrt{2}e^{i\theta_1}/|s_2^2+it_2|^{1/2} \\
1 \end{matrix}\right).
\end{eqnarray*}
We can easily verify that ${\bf q}_+$ (resp. ${\bf q}_-$) lies on the isometric sphere of $B$ (resp. $B^{-1}$). 
A calculation shows that
\begin{eqnarray*}
B{\bf q}_+ & = & \left(\begin{matrix}
-1/|s_2^2+it_2|-2s_2(s_2^2+it_2)e^{i\theta_1-i\theta_2}/|s_2^2+it_2|^{5/2}-(s_2^2-it_2)/|s_2^2+it_2|^2 \\
\sqrt{2}s_2e^{i\theta_2}/|s_2^2+it_2|+\sqrt{2}(s_2^2+it_2)^2e^{i\theta_1}/|s_2^2+it_2|^{5/2} \\
(s_2^2-it_2)/|s_2^2+it_2|
\end{matrix}\right) \\
& = & \frac{s_2^2-it_2}{|s_2^2+it_2|} \left(\begin{matrix}
-1/|s_2^2+it_2|-2s_2(s_2^2+it_2)^2e^{i\theta_1-i\theta_2}/|s_2^2+it_2|^{7/2}-(s_2^2+it_2)/|s_2^2+it_2|^2 \\
-\sqrt{2}s_2(s_2^2+it_2)e^{i\theta_2}/|s_2^2+it_2|^2-\sqrt{2}(s_2^2+it_2)^3e^{i\theta_1}/|s_2^2+it_2|^{7/2} \\
1 \end{matrix}\right).
\end{eqnarray*}
This shows that, $B$ projectively maps ${\bf q}_+$ to ${\bf q}_-$ if and only if $(s_2^2+it_2)^3=|s_2^2+it_2|^3$
which is equivalent to $t_2=0$ as $s_2^2\ge 0$.
Similarly, $A{\bf q}_-={\bf q}_+$ if and only if both
\begin{eqnarray*}
(s_1^2-it_1)|s_2^2+it_2|^{1/2} & = & 2s_1+\frac{2s_1s_2(s_2^2+it_2)e^{i\theta_2-i\theta_1}}{|s_2^2+it_2|^{3/2}}
+\frac{4s_2it_2e^{i\theta_1-\theta_2}}{|s_2^2+it_2|^2}-\frac{2it_2}{|s_2^2+it_2|^{3/2}},   \\
s_1|s_2^2+it_2|^{1/2} &= & 2+\frac{2s_2^3e^{i\theta_2-i\theta_1}}{|s_2^2+it_2|^{3/2}}.
\end{eqnarray*}

Thus, ${\bf q}_-$ is mapped to ${\bf q}_+$ by $A$ only when $t_1=t_2=0$, $\theta_1=\theta_1$ and $s_1s_2=4$. 

\medskip

\subsection{Proof of Theorem \ref{thm-main0}(4)}

We now construct a fundamental domain for $\langle B\rangle$ bounded by two fans. To do this, we apply the involution
$\iota$ given by \eqref{eq-iota}. 

Fix $k\in{\mathbb R}-\{0\}$. Consider the fan with with vertex the origin $o$ given by 
$$
F_k^{(o)}=\left\{\left(\begin{matrix} 1 \\ \sqrt{2}(k+ia) \\ -k^2-a^2+ib-2ika \end{matrix}\right)\ :\ (a,b)\in{\mathbb R}^2\right\}.
$$
Putting this into Heisenberg coordinates $(x+iy,v)$ we find
$$
x=\frac{-k(k^2+a^2)+a(b-2ka)}{(k^2+a^2)^2+(b-2ka)^2},\quad y=\frac{-a(k^2+a^2)-k(b-2ka)}{(k^2+a^2)^2+(b-2ka)^2},\quad
v=\frac{-b+2ka}{(k^2+a^2)^2+(b-2ka)^2}.
$$

\begin{lemma}\label{lem-project-fan-o}
Suppose that $k>0$. The images of $F_k^{(o)}$ and $F_{-k}^{(o)}$ 
under vertical projection are, respectively, the sets given in polar coordinates $(r,\theta)$ by
$$
r\le \frac{1-\cos(\theta)}{2k}\quad\hbox{ and }\quad r\le\frac{1+\cos(\theta)}{2k}.
$$
In Cartesian coordinates, these are 
$$
y^2-4kx(x^2+y^2)-4k^2(x^2+y^2)^2\ge 0\quad \hbox{ and }\quad y^2+4kx(x^2+y^2)-4k^2(x^2+y^2)^2\ge 0.
$$
These are the interiors of two cardioids.
\end{lemma}

\begin{proof}
We want to find out the region in the $(x,y)$ plane which is the image of $F_k^{(o)}$ under vertical projection.
To do this, we consider local coordinates $(a,b)$ on $F_k^{(o)}$ and we seek points where the map 
$(a,b)\longmapsto (x,y)$ is not a local diffeomorphism. In other words, we need to find where the Jacobian $J$ of this map
vanishes. A calculation shows this Jacobian is
$$
J=\frac{-a(k^2+a^2)+k(b-2ka)}{((k^2+a^2)^2+(b-2ka)^2)^2}.
$$
Therefore, the Jacobian vanishes precisely when
$$
b=2ka+\frac{a}{k}(k^2+a^2).
$$
We can parameterise this curve using $a$ as:
$$
x=\frac{-k(k^2-a^2)}{(k^2+a^2)^2},\quad y=\frac{-2k^2a}{(k^2+a^2)^2},\quad v=\frac{-ka}{(k^2+a^2)^2}.
$$
Writing it in polar coordinates $x+iy=re^{i\theta}$ and eliminating $a$ gives 
$$
r^2=x^2+y^2=\frac{k^2}{(k^2+a^2)^2},\quad
x=r\cos(\theta)=\frac{k}{(k^2+a^2)}\left(1-\frac{2k^2}{(k^2+a^2)}\right).
$$
Using $k>0$ we see that $r=k/(k^2+a^2)$ and $2kr=1-\cos(\theta)$.

Similarly, in terms of $x$ and $y$ the region is given by
$$
y^2-4kx(x^2+y^2)-4k^2(x^2+y^2)^2
=\left(\frac{a(k^2+a^2)-k(b-2ka)}{(k^2+a^2)^2+(b-2ka)^2}\right)^2\ge 0.
$$

A similar argument holds for $F_{-k}^{(o)}$ by replacing $k$ with $-k$ throughout. Only here we have 
$$
x=r\cos(\theta)=\frac{k}{k^2+a^2}\left(\frac{2k^2}{k^2+a^2}-1\right)
$$
and so $2kr=1+\cos(\theta)$.
\end{proof}

\medskip

\begin{figure}[htbp]
  \includegraphics[width=0.7\textwidth]{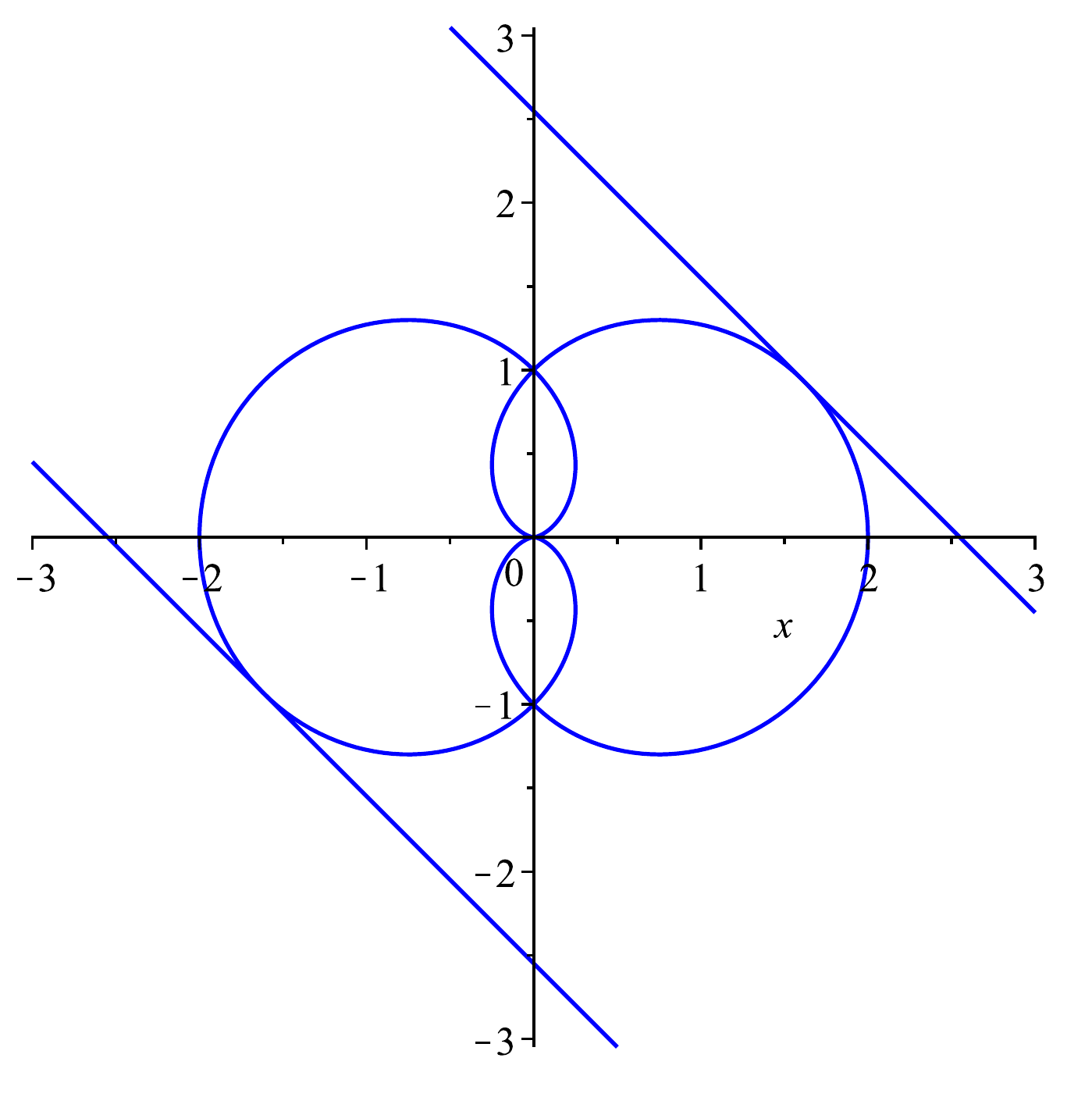}
  \caption{Cardiods which are the vertical projection of finite fans bounding a fundamental domain for $\langle B\rangle$
and strip which is the vertical projection of a fundamental domain for $\langle A\rangle$ bounded by two infinite fans. 
We need the cardioids to both lie in the strip.}\label{fig-2}
\end{figure}

For $k>0$ let $C_k$ and $C_{-k}$ be the cardioids given in polar coordinates $(r,\theta)$ and Cartesian coordinates $(x,y)$ by
\begin{eqnarray*}
C_k & = & \left\{(r,\theta)\ : \ r=\frac{1-\cos(\theta)}{2k}\right\}
=\left\{(x,y)\ :\  y^2-4kx(x^2+y^2)-4k^2(x^2+y^2)^2=0\right\}, \\
C_{-k} & = & \left\{(r,\theta)\ : \ r=\frac{1+\cos(\theta)}{2k}\right\}
=\left\{(x,y)\ :\  y^2+4kx(x^2+y^2)-4k^2(x^2+y^2)^2=0\right\}.
\end{eqnarray*}

\begin{lemma}\label{lem-cardiod-in-strip}
For each $k>0$ and each $\phi$ with $-\pi/2\le \phi\le \pi/2$, the cardioids $C_k$ and $C_{-k}$ are both contained 
in the strip $S_\phi$ given by
$$
S_{k,\phi}=\left\{(x,y)\ :\ \Bigl|x\cos(\phi)+y\sin(\phi)\Bigr| \le \frac{\cos^3(\phi/3)}{k}\right\}.
$$
\end{lemma}

\begin{proof}
(See Figure~\ref{fig-2}.) 
The cardioid $C_k$ is singular at the origin $r=0$. Using the polar parametrisation of $C_k$ we can parametrise its non-singular
points by $\theta\in(0,2\pi)$. A point on $C_k$ are given in parametric representation by
$$
\bigl(x(\theta),\,y(\theta)\bigr) = \left(\frac{(1-\cos(\theta))\cos(\theta)}{2k},\,\frac{(1-\cos(\theta))\sin(\theta)}{2k}\right).
$$
Consider
$$
f_\phi(\theta)=k\bigl(x(\theta)\cos(\phi)+y(\theta)\sin(\phi)\bigr)=\frac{(1-\cos(\theta))\cos(\theta-\phi)}{2}.
$$
Differentiating with respect to $\theta$ we have
\begin{eqnarray*} 
f'_\phi(\theta) & = & \frac{-(1-\cos(\theta))\sin(\theta-\phi)+\sin(\theta)\cos(\theta-\phi)}{2} \\
& = & \sin(\theta/2)\cos(3\theta/2-\phi).
\end{eqnarray*}
Since $0<\theta/2<\pi$ we see $\sin(\theta/2)\neq 0$ and so the maximum and minimum values of 
$f_\phi(\theta)$ occur when $3\theta/2-\phi=\pi/2+n\pi$ where $n$ is an integer.  That is, $\theta=2\phi/3+\pi/3+2n\pi/3$.
We have
$$
f_\phi(2\phi/3+\pi/3+2n\pi/3)=\cos^2(\phi/3+(n-1)\pi/3)\cos(\phi/3-(2n+1)\pi/3).
$$
Since $-\pi/2\le\phi\le\pi/2$ we see that the maximum value of $|f_\phi(\theta)|$ occurs when $n=1$, that is
when $\theta=2\phi/3+\pi$. The maximum value is
$$
|f_{\phi}(2\phi/3+\pi)|=\cos^3(\phi/3).
$$

Similarly for $C_{-2k}$. Here we need to maximise the absolute value of
$$
g_\phi(\theta)=\frac{(1+\cos(\theta))\cos(\theta-\phi)}{2}.
$$
This occurs when $\theta=2\phi/3$. The maximum value is
$$
|g_{\phi}(2\phi/3)|=\cos^3(\phi/3).
$$
\end{proof}

\medskip

If we have equality in Lemma~\ref{lem-cardiod-in-strip}, we see that the points on the fans $F_k^{(o)}$ and $F_{-k}^{(o)}$
that project to points on the boundary of the strip are, respectively:
\begin{equation}\label{eq-extreme}
\left(\begin{matrix} 
-\cos^3(\phi/3)e^{i\phi/3}/k^2 \\
-\sqrt{2}\cos^2(\phi/3)e^{2i\phi/3}/k \\
1
\end{matrix}\right), \quad
\left(\begin{matrix} 
-\cos^3(\phi/3)e^{i\phi/3}/k^2 \\
\sqrt{2}\cos^2(\phi/3)e^{2i\phi/3}/k \\
1
\end{matrix}\right).
\end{equation}

\begin{proof}(Theorem~\ref{thm-main0}(4))
We now prove the condition (4) of Theorem~\ref{thm-main0}.
Conjugating $A$ and $B$ by a Heisenberg rotation fixing $o$ and $\infty$, if necessary, we assume that $\theta_2=0$.

It is clear that (independent of $t_1$) the map $A$ sends $F_{{-\frac{s_1}{2}}e^{i\theta_1}}^{(\infty)}$ to 
$F_{{\frac{s_1}{2}}e^{i\theta_1}}^{(\infty)}$:
$$
\left(\begin{matrix} 1 & -\sqrt{2}s_1e^{-i\theta_1} & -s_1^2+it_1 \\ 0 & 1 & \sqrt{2}s_1e^{i\theta_1} \\ 0 & 0 & 1 \end{matrix}\right)
\left(\begin{matrix} -\frac{s_1^2}{4}-a^2+ib-is_1a \\ \sqrt{2}(-\frac{s_1}{2}+ia)e^{i\theta_1} \\ 1 \end{matrix}\right)
=\left(\begin{matrix} -\frac{s_1^2}{4}-a^2+ib-ias_1+it_1 \\ \sqrt{2}(\frac{s_1}{2}+ia)e^{i\theta_1} \\ 1 \end{matrix}\right)
$$
Therefore the slab bounded by these two fans is a fundamental region for $\langle A\rangle$. This slab is
$$
D_A=\left\{(x+iy,v)\in {\mathfrak N}\ : \ \bigl| x\cos(\theta_1)+y\cos(\theta_1)\bigr|\le \frac{s_1}{2} \right\}\cup\{\infty\}.
$$
The image of $D_A$ under vertical projection is the strip 
$$
S_A=\left\{(x,y)\ :\ 
\bigl| x\cos(\theta_1)+y\cos(\theta_1)\bigr|\le \frac{s_1}{2} \right\}.
$$

Similarly, (independent of $t_2$) the map $B$ sends the fan $F_{-\frac{s_2}{2}}^{(o)}$ to the fan $F_{\frac{s_2}{2}}^{(o)}$:
$$
\left(\begin{matrix} 1 & 0 & 0 \\ \sqrt{2}s_2 & 1 & 0 \\ -s_2^2+it_2 & -\sqrt{2}s_2 & 1 \end{matrix}\right)
\left(\begin{matrix} 1 \\ \sqrt{2}(-\frac{s_2}{2}+ia) \\ -\frac{s_2^2}{4}-a^2+ib-is_1a \end{matrix}\right)
=\left(\begin{matrix} 1 \\ \sqrt{2}(\frac{s_2}{2}+ia) \\ -\frac{s_2^2}{4}-a^2+ib -ias_2+it_2\end{matrix}\right).
$$
That is, $\iota B\iota$ sends the fan $F_{-\frac{s_2}{2}}^{(\infty)}$ to the fan $F_{\frac{s_2}{2}}^{(\infty)}$.
So the slab $D_{\iota B\iota}$ bounded by these two fans is a fundamental domain for $\langle \iota B\iota\rangle$, where
$$
D_{\iota B\iota}=\left\{(x+iy,v)\in{\mathfrak N}\ :\ 
\bigl| x\bigr|\le \frac{s_2}{2} \right\}\cup\{\infty\}.
$$
Thus the image of $D_{\iota B\iota}$ under $\iota$ is a fundamental domain for $\langle B\rangle$. This consists
of all points in the exterior of the fans $F_{-\frac{s_2}{2}}^{(o)}$ and $F_{\frac{s_2}{2}}^{(o)}$.

In order to show that $\langle A,B\rangle$ is free it is sufficient to show that the fans $F_{\pm \frac{s_2}{2}}^{(o)}$ are contained 
in the slab $D_A$ between $F_{{-\frac{s_1}{2}}e^{i\theta_1}}^{(\infty)}$ to $F_{{\frac{s_1}{2}}e^{i\theta_1}}^{(\infty)}$. 
It suffices to show that the vertical projections 
of  $F_{\pm \frac{s_2}{2}}^{(o)}$ are contained in the strip $S_A$. The vertical projections of these fans are given in 
polar coordinates $(r,\theta)$ by
$$
r\le \frac{1-\cos(\theta)}{s_2},\quad  r\le \frac{1-\cos(\theta)}{-s_2}.
$$
Using Lemma~\ref{lem-cardiod-in-strip} with $\phi=\theta_1$ and $k=s_2/2$  (see Figure~\ref{fig-2}) we see that this is true
provided
$$
\frac{\cos^3(\frac{\theta_1}{3})}{s_2/2}\le \frac{s_1}{2}.
$$
The result follows.
\end{proof}

\medskip

Plugging in $k=s_2/2$ and $\phi=\theta_1$ into the formulae for the two extreme points \eqref{eq-extreme}, 
we see that $B$ maps one to the other
if and only if $\theta_1=0$ and $t_2=0$. Similarly, $A$ maps one to the other if and only if $s_1s_2=4$, $\theta_1=0$ and $t_1=0$.

\end{document}